\documentclass[11pt]{article}
\usepackage[utf8]{inputenc}
\usepackage{mathtools}
\usepackage{amsmath}
\usepackage{amsthm}
\usepackage{amsfonts}
\usepackage{amssymb}
\usepackage{graphicx}
\usepackage[usenames,dvipsnames]{color}
\usepackage{hyperref}
\usepackage[left=2cm,right=2cm,top=2cm,bottom=2cm]{geometry}
\usepackage{dsfont}
\usepackage{enumerate}

\usepackage[capitalize,nameinlink]{cleveref}
\crefname{theorem}{Theorem}{Theorems}
\crefname{lemma}{Lemma}{Lemmas}
\crefname{claim}{Claim}{Claims}
\crefname{prop}{Proposition}{Propositions}

\theoremstyle{plain}
\newtheorem{theorem}{Theorem}[section]

\newtheorem{proposition}[theorem]{Proposition}
\newtheorem{lemma}[theorem]{Lemma}

\newtheorem{conjecture}[theorem]{Conjecture}

\theoremstyle{definition}

\newtheorem{def/prop}[theorem]{Definition/Proposition}

\newcommand{\lpr}[1]{\left(#1\right)}

\newcommand{\cE}{\mathcal{E}}

\newcommand{\cL}{\mathcal{L}}

\renewcommand{\Pr}{\mathbb{P}}

\usepackage[
backend=biber,
style=numeric,
maxnames=99
]{biblatex}
\renewbibmacro{in:}{%
  \ifentrytype{article}{}{\printtext{\bibstring{in}\intitlepunct}}}

\title{\scshape On Off-diagonal $F$-Ramsey numbers}

\author{Sammy Luo \thanks{Massachusetts Institute of Technology. Research supported by NSF Award No. 2303290. \texttt{sammyluo@mit.edu}}
\and
Zixuan Xu \thanks{Massachusetts Institute of Technology. \texttt{zixuanxu@mit.edu}}
}

\begin{document}

\maketitle

\begin{abstract}
A graph is $(t_1, t_2)$-Ramsey if any red-blue coloring of its edges contains either a red copy of $K_{t_1}$ or a blue copy of $K_{t_2}$. The size Ramsey number is the minimum number of edges contained in a $(t_1,t_2)$-Ramsey graph. Generalizing the notion of size Ramsey numbers, the $F$-Ramsey number $r_F(t_1, t_2)$ is defined to be the minimum number of copies of $F$ in a $(t_1,t_2)$-Ramsey graph. It is easy to see that $r_{K_s}(t_1,t_2)\le \binom{r(t_1,t_2)}{s}$. Recently, Fox, Tidor, and Zhang showed that equality holds in this bound when $s=3$ and $t_1=t_2$, i.e. $r_{K_3}(t,t) = \binom{r(t,t)}{3}$. They further conjectured that $r_{K_s}(t,t)=\binom{r(t,t)}{s}$ for all $s\le t$, in response to a question of Spiro. 

In this work, we study the off-diagonal variant of this conjecture: is it true that $r_{K_s}(t_1,t_2)=\binom{r(t_1,t_2)}{s}$ whenever $s\le \max(t_1,t_2)$? Harnessing the constructions used in the recent breakthrough work of Mattheus and Verstraëte on the asymptotics of $r(4,t)$, we show that when $t_1$ is $3$ or $4$, the above equality holds up to a lower order term in the exponent.
\end{abstract}

\section{Introduction}

We say that a graph $G$ is \emph{$(t_1,t_2)$-Ramsey} (or \emph{$t$-Ramsey} if $t_1=t_2=t$) if every $2$-coloring of its edges in red and blue results in either a red copy of $K_{t_1}$ or a blue copy of $K_{t_2}$. The smallest $n$ such that $K_n$ is $(t_1,t_2)$-Ramsey is called the Ramsey number $r(t_1,t_2)$ (or $r(t)$ in the diagonal case). A closely related quantity is the \emph{size Ramsey number} $\hat{r}(t_1,t_2)$, defined to be the minimum number of edges contained in a $(t_1,t_2)$-Ramsey graph. A classical observation by Chv\'atal gives the relationship between the usual Ramsey number and size Ramsey number for complete graphs: For any $t_1,t_2\ge 1$,
\[
    \hat{r}(t_1,t_2)=\binom{r(t_1,t_2)}{2}.
\]
A recent paper of Fox, Tidor, and Zhang \cite{fox2023triangle} initiated the study of a new variant of the Ramsey number: For a given graph $F$, the $F$-Ramsey number $r_F(t_1,t_2)$ is the smallest number of copies of $F$ contained in a graph that is $(t_1,t_2)$-Ramsey. This definition simultaneously generalizes the usual and size Ramsey numbers: taking $F=K_1$ gives the usual Ramsey number, while taking $F=K_2$ yields the size Ramsey number. 

In \cite{fox2023triangle}, the authors addressed a natural question of Sam Spiro asking whether an analogue of Chv\'atal's result holds for choices of $F$ other than $K_2$. They showed that 
\[
r_{K_3}(t,t)=\binom{r(t)}{3},
\]
for all sufficiently large $t$. They conjectured that for all $s\le t$,
\begin{equation}\label{eq:conj-diagonal}
    r_{K_s}(t,t)=\binom{r(t)}{s}, \quad \forall s\le t.
\end{equation}

Their conjecture naturally extends to the off-diagonal setting.
\begin{conjecture}
    \label{conj:offdiag}
    For all positive integers $s,t_1,t_2$ satisfying $s\le \max(t_1,t_2)$, 
    \[r_{K_s}(t_1,t_2) = \binom{r(t_1,t_2)}{s}.\]
\end{conjecture}
It is easy to see that the upper bound in the conjectured equality is true, since $K_{r(t_1,t_2)}$ is $(t_1,t_2)$-Ramsey. The conjecture is trivially true when $s = 1$, and the case $s=2$ is Chv\'atal's observation. The main result of \cite{fox2023triangle} corresponds to the case where $s=3$ and $t_1=t_2$ is sufficiently large, but the same argument shows that \cref{conj:offdiag} is true for $s = 3$ in general. Unfortunately, their proof technique cannot be directly extended to $s \ge 4$, and it appears that significantly different strategies would be required to tackle the conjecture in full generality.

In this paper, we investigate \cref{conj:offdiag} in the simplest off-diagonal cases $t_1=3$ and $t_1=4$ (the cases $t_1\le 2$ are trivial). In these cases, the asymptotic value of the usual Ramsey number is almost fully determined: for $r(3,t)$, results of Ajtai, Koml\'os, and Szemer\'edi \cite{AKS80} and of Kim \cite{KimR3t} combine to yield $r(3,t)=\Theta(t^2 / \log t)$ as $t\to \infty$; for $r(4,t)$, results of Ajtai, Koml\'os, and Szemer\'edi \cite{AKS80} and of Mattheus and Verstraëte \cite{mattheus2024} combine to yield
\[C\cdot \frac{t^3}{\log^4 t} \le r(4,t) \le (1+o(1))\frac{t^3}{\log^2 t},\]
for some absolute constant $C> 0$ as $t\to \infty$, determining the asymptotics of $r(4,t)$ up to logarithmic factors. Using these known bounds on the usual Ramsey number, we obtain the following results indicating that \cref{conj:offdiag} holds in these cases up to a lower order term in the exponent. 

\begin{theorem}
    \label{thm:rs3t}
    For all $s\le t$,
    \[
        r_{K_s}(3,t)=\binom{\Omega\left(\frac{t^2}{\log^2(t)}\right)}{s}.
    \]
    In particular, 
    \[r_{K_s}(3,t)=\binom{r(3,t)}{s}^{1-o(1)}.\]
\end{theorem}
\begin{theorem}
    \label{thm:rs4t}
    For all sufficiently large $t$ and all $s\le t$,
    \[
        r_{K_s}(4,t)=\binom{\Omega\left(\frac{t^3}{\log^4(t)}\right)}{s}.
    \]
    In particular, 
    \[r_{K_s}(4,t)=\binom{r(4,t)}{s}^{1-o(1)}.\]
\end{theorem}
In both of these results, the implicit constant in the $\Omega$ notation is an absolute constant, and the $o(1)$ term in the exponent goes to zero as $t$ grows, with no dependence on $s$.

Our results provide evidence towards \cref{conj:offdiag} being at least asymptotically true. In contrast to the arguments in \cite{fox2023triangle}, our methods work equally well for all $s\le t$.

Our arguments utilize the construction given in \cite{mattheus2024} of a $K_4$-free graph $H$ with few independent sets of size roughly $t$ (and an analogous construction of a $K_3$-free graph). The key idea is to produce an edge-coloring contradicting the Ramsey property for a given graph from a random homomorphism into $H$, which then yields a lower bound on the number of copies of $K_t$ in the relevant Ramsey graph. The Kruskal-Katona theorem then converts this into an appropriate lower bound on the number of copies of $K_s$ for each $s\le t$. 

The paper is organized as follows. In \cref{sec:complete-kst}, we present the proofs of our main results on off-diagonal $K_s$-Ramsey numbers. We first prove \cref{thm:rs4t} in \cref{sec:rs4t}, then use a similar argument in \cref{sec:rs3t} to prove \cref{thm:rs3t}. Finally, in \cref{sec:cycle}, we discuss how our strategy can be applied to the more general setting of $F$-Ramsey numbers of $(H_1,H_2)$-Ramsey graphs.

\section{Off-diagonal $K_s$-Ramsey Numbers}\label{sec:complete-kst}

This section contains our proofs of the lower bounds on both $r_{K_s}(3,t)$ and $r_{K_s}(4,t)$. While the arguments for the two settings are of similar complexity, we will begin by considering the case of $(4,t)$-Ramsey graphs, because the construction of Mattheus and Verstra\"ete that we will utilize for this case, while somewhat more intricate than the construction for the case of $(3,t)$-Ramsey graphs, is more explicitly described in~\cite{mattheus2024}.

\subsection{Asymptotics of $r_{K_s}(4,t)$}\label{sec:rs4t}

We use the following result of Mattheus and Verstraëte, which yields a construction of a family of $K_4$-free graphs with few independent sets of size $t$, for all sufficiently large $t$.

\begin{theorem}[\cite{mattheus2024}]
    \label{thm:r4t}
    For each prime power $q \ge 2^{40}$, there exists a $K_4$-free graph $H$ with $q^2(q^2-q+1)$ vertices such that for every set $X$ of at least $m = 2^{24}q^2$ vertices of $H$,
    \[
    e(H [X]) \ge \frac{|X|^2}{256q}.
    \]
\end{theorem}

The following result is used in \cite{mattheus2024} to give an upper bound on the number of independent sets of a fixed size $t=2^{30}q\log^2 q$ in the graph $H$ given by \cref{thm:r4t}. We will use it to bound the number of independent sets of a certain range of sizes.

\begin{proposition}[\protect{\cite[Proposition 4]{mattheus2024}}]
    \label{prop:countindeps}
    Let $H$ be a graph on $n$ vertices, and let $r,R\in \mathbb{N}$, and $\alpha\in [0,1]$ satisfy:
    \[
    e^{-\alpha r}n\le R,
    \]
    and, for every subset $X\subseteq V(H)$ of at least $R$ vertices,
    \[
    2e(X)\ge \alpha|X|^2.
    \]
    Then for any $t\ge r$, the number of independent sets of size $t$ in $H$ is at most
    \[
    \binom{n}{r}\binom{R}{t-r}.
    \]
\end{proposition}

We also make use of the following form of the Kruskal-Katona theorem due to Lov\'{a}sz.

\begin{lemma}[\protect{\cite[Exercise 31(b)]{Lov93}}]\label{lem:kruskal-katona}
Given a set $X$ and positive integers $s\le t$, let $A$ be a set of $t$-element subsets of $X$, and let $B$ be the set of all $s$-element subsets of the sets in $A$. If $|A|=\binom{n}{t}$, then $|B|\ge \binom{n}{s}$.
\end{lemma}

In the setting of graphs, the Kruskal-Katona theorem implies that if a graph contains $\binom{n}{t}$ copies of $K_t$, then it must contain at least $\binom{n}{s}$ copies of $K_s$ for $s\le t$. Now we are ready to prove \cref{thm:rs4t}

\begin{proof}[Proof of \cref{thm:rs4t}]
We first prove the theorem for the case where $s=t$.

Let $q\ge 2^{40}$ be a power of $2$ such that $t\in [2^{33}q\log^2 q,2^{33}(2q)\log^2 (2q)]$. Such a choice of $q$ exists because the union of the relevant intervals contains all sufficiently large integers.

Let $H$ be the graph given by \cref{thm:r4t} on $n:=q^2(q^2-q+1)$ vertices. Then we can apply \cref{prop:countindeps} with $R=2^{24}q^2$, $r=2^{10}q\log q$, and $\alpha=1/(2^{8}q)$ to show that, for any $t'\in [t/8,t]$, the number of independent sets of size $t'$ in $H$ is at most
\[
\binom{n}{r}\binom{R}{t'-r}\le n^r \binom{R}{t'} \le  q^{4r}\left(\frac{eR}{t'}\right)^{t'} \le (q/\log^2 q)^{t'}. 
\]
Here we have used the fact that by our choices of $t,n,R,r,t'$, we have $n\le q^4$, $r\le t' \le R/2$, $q^{4r}\le 2^{t'}$, and $2eR/t'\le q/\log^2 q$.

Let $G$ be a $(4,t)$-Ramsey graph containing exactly $N :=r_{K_t}(4,t)$ copies of $K_t$. Our goal is to prove a lower bound on $N$. We color the edges of $G$ as follows: Take a uniformly random map $\pi:V(G)\to V(H)$. For an edge $(v,w)\in G$, color it red if $(\pi(v),\pi(w))\in E(H)$, and blue otherwise (including the case where $\pi(v)=\pi(w)$).

Since $H$ is $K_4$-free by construction, there are no red $K_4$'s in $G$. For any copy $K$ of $K_t$ in $G$, it can be monochromatically blue only if its image $\pi(K)$ is an independent set in $H$. The probability that this occurs is bounded above by
\begin{align*}
    \rho &:=\sum_{i=1}^{t} \Pr[\pi(K)\text{ is an independent set of size }i] \\
    &\le \sum_{i=1}^{t/8} \binom{n}{i} \left(\frac{i}{n}\right)^t + \sum_{i=t/8+1}^t (q/\log^2 q)^{i} \left(\frac{i}{n}\right)^t \\
    &\le 2n^{-(t-t/8)} t^t + 2(q/\log^2 q)^t \left(\frac{t}{n}\right)^t \\
    &= O(t^t q^{-3.5t} + t^t (q^3 \log^2 q)^{-t}) \\
    &= O(t^t (q^3 \log^2 q)^{-t}).
\end{align*}
To show the first inequality, we trivially bound the probability that $|\pi(K)|\le i$ for each $i\le t/8$, then use the upper bound derived earlier for the number of independent sets of each size $i\in [t/8,t]$. The second-to-last line uses the fact that $n=\Theta(q^4)$.

The expected number of blue copies of $K_t$ in our coloring of $G$ is at most $N\rho$. If $N\rho<1$, then there is such a coloring with no red $K_4$ or blue $K_t$, contradicting the fact that $G$ is $(4,t)$-Ramsey. So, since $t=\Theta(q\log^2 q)$, and thus $q=\Theta(\frac{t}{\log^2 t})$, we have
\[
N\ge \frac{1}{\rho} =\Omega \left(\left( t^{-1} \left(q^3\log^2 q\right)\right)^t\right) = \frac{\left(\Omega(\frac{t^3}{\log^4 t})\right)^t}{t^t} =  \binom{\Omega(\frac{t^3}{\log^4(t)})}{t},
\]
using the estimate $\binom{n}{k}\le (\frac{en}{k})^k$. 
This shows the desired lower bound on $r_{K_t}(4,t)$.

By the Kruskal-Katona theorem, since $G$ contains $\binom{\Omega(\frac{t^3}{\log^4(t)})}{t}$ copies of $K_t$, it must contain at least $\binom{\Omega(\frac{t^3}{\log^4(t)})}{s}$ copies of $K_s$ for each $s\le t$. 

Finally, since $r(4,t)\le (1+o(1))\frac{t^3}{\log^2 t}=\left(\frac{t^3}{\log^4 t}\right)^{1+o(1)}$, we have $r_{K_s}(4,t)=\binom{\Omega(\frac{t^3}{\log^4(t)})}{s}\ge \binom{r(4,t)}{s}^{1-o(1)}$. Since $K_{r(4,t)}$ is $(4,t)$-Ramsey by definition, we have a trivial matching upper bound of $r_{K_s}(4,t)\le \binom{r(4,t)}{s}$. This concludes the proof.

\end{proof}

\subsection{Asymptotics of $r_{K_s}(3,t)$}\label{sec:rs3t}

The asymptotics for $r(3,t)$ have been thoroughly studied. It is known that \[r(3,t)=\Theta\left(\frac{t^2}{\log t}\right),\] where the upper bound was shown by Ajtai, Koml\'os, and Szemer\'edi \cite{AKS80} and the lower bound by Kim \cite{KimR3t}. This lower bound was initially proven by constructing triangle-free graphs on $n$ vertices with chromatic number $\Omega(\sqrt{n/\log n})$ using the R\"odl nibble method. However, a different construction, given by Mattheus and Verstraëte \cite[Section~3]{mattheus2024} as an adaptation of their own construction for $r(4,t)$, is more suitable for our approach.

\begin{theorem}[\cite{mattheus2024}]\label{thm:r3t-H}
    There exists a constant $C>0$ for which the following holds. For each prime power $q\ge 2^{40}$, there exists a triangle-free graph $H$ with $q^3+q^2+q+1$ vertices such that for every set $X\subseteq V(H)$ of size $|X|\ge C\cdot q^2$,
    \[e(H[X])\ge \frac{|X|^2}{C\cdot q}.\]
\end{theorem}

The proof of \cref{thm:rs3t} uses the same strategy as the proof of \cref{thm:rs4t}. However, different choices of parameters are used in this proof, so we include the full proof for the sake of completeness.

\begin{proof}[Proof of \cref{thm:rs3t}]
    As in the proof of \cref{thm:rs4t}, we first prove the theorem in the case $s=t$. Let $C$ be the constant from \cref{thm:r3t-H}, and let $q\ge 2^{40}$ be a power of $2$ such that $t \in [2^8Cq\log^2q, 2^8C(2q)\log^2(2q)]$.

    Let $H$ be the graph given by \cref{thm:r3t-H} on $n := q^3+q^2+q+1$ vertices. Applying \cref{prop:countindeps} with the parameters $R = 2^5Cq^2, r = 2Cq\log q$ and $\alpha = 1/(C\cdot q)$ gives that for any $t'\in [t/8, t]$, the number of independent sets of size $t'$ in $H$ is at most 
    \[
    \binom{n}{r}\binom{R}{t'-r}\le n^r \binom{R}{t'} \le  (2q^{3})^r\left(\frac{eR}{t'}\right)^{t'} \le (q/\log^2 q)^{t'}. 
    \]
    Here we have used the facts that $n\le 2q^3$, $r\le t' \le R/2$, $(2q^{3})^r\le 2^{t'}$, and $2eR/t'\le q/\log^2 q$.

    Let $G$ be a $(3,t)$-Ramsey graph containing exactly $N :=r_{K_t}(3,t)$ copies of $K_t$. Consider the following coloring of $G$: Take a uniformly random map $\pi:V(G)\to V(H)$. For an edge $(v,w)\in G$, color it red if $(\pi(v),\pi(w))\in E(H)$, and blue otherwise (including the case where $\pi(v)=\pi(w)$).

    Since $H$ is taken to be triangle-free, there are no red triangles in $G$. For any copy $K$ of $K_t$ in $G$, it can be monochromatically blue only if its image $\pi(K)$ is an independent set in $H$. The probability that this occurs is bounded above by
    \begin{align*}
        \rho &:=\sum_{i=1}^{t} \Pr[\pi(K)\text{ is an independent set of size }i] \\
        &\le \sum_{i=1}^{t/8} \binom{n}{i} \left(\frac{i}{n}\right)^t + \sum_{i=t/8+1}^t (q/\log^2 q)^{i} \left(\frac{i}{n}\right)^t \\
        &\le 2n^{-(t-t/8)} t^t + 2(q/\log^2 q)^t \left(\frac{t}{n}\right)^t \\
        &= O(t^t q^{-21t/8} + t^t (q^2 \log^2 q)^{-t}) \\
        &= O(t^t (q^2 \log^2 q)^{-t}),
    \end{align*}
    by the same reasoning as in the proof of \cref{thm:rs4t}. 
    Since the expected number of blue copies of $K_t$ in this coloring is at most $N\rho$, we must have $N\rho \ge 1$ because otherwise there exists is a coloring with no red $K_4$ or blue $K_t$, contradicting the fact that $G$ is $(3,t)$-Ramsey. So, by the same reasoning as in the proof of \cref{thm:rs4t}, we have
    \[
    N\ge \frac{1}{\rho} =\Omega \left(\left( t^{-1} \left(q^2\log^2 q\right)\right)^t\right) =  \binom{\Omega(\frac{t^2}{\log^2(t)})}{t},
    \]
    showing the lower bound on $r_{K_t}(3,t)$ as desired.
    
    By the Kruskal-Katona theorem, since $G$ contains $\binom{\Omega(\frac{t^2}{\log^2(t)})}{t}$ copies of $K_t$, it must contain at least $\binom{\Omega(\frac{t^2}{\log^2(t)})}{s}$ copies of $K_s$ for each $s\le t$. Combining this with the fact that $r(3,t)=\Theta(t^2/\log t)$, as well as the trivial upper bound $K_s(3,t)\le \binom{r(3,t)}{s}$, we arrive at the second equation in the Theorem statement as before. This concludes the proof.
\end{proof}

\section{Cycle-complete $K_s$-Ramsey numbers}\label{sec:cycle}

The notion of $(t_1,t_2)$-Ramsey graphs extends to a more general setting, where monochromatic subgraphs other than cliques are considered: For any graphs $H_1$ and $H_2$, we say a graph $G$ is $(H_1,H_2)$-Ramsey if every $2$-coloring of its edges in red and blue results in either a red copy of $H_1$ or a blue copy of $H_2$. The definition of the $F$-Ramsey number can likewise be generalized by letting $r_F(H_1,H_2)$ be the smallest number of copies of $F$ contained in an $(H_1,H_2)$-Ramsey graph. The diagonal case of this more general definition is briefly discussed in \cite{fox2023triangle}. There, the authors point out that $r_F(H,H)=0$ in the following two cases:
\begin{itemize}
    \item $m_2(F)>m_2(H)$ (as shown in \cite{RR1995}), where we define $m_2(G)=\max_{G'\subseteq G}\frac{e(G')-1}{|G'|-2}$;
    \item The girth of $F$ is smaller than the girth of $H$ (equivalent to a result in \cite{reiher2023girth}).
\end{itemize}

One can ask whether an analogue of \cref{conj:offdiag} holds for $r_{K_s}(H_1,H_2)$, for some choices of $(H_1,H_2)$ other than a pair of cliques. Namely, is it true that
\[r_{K_s}(H_1,H_2) = \binom{r(H_1,H_2)}{s},\]
for some appropriate range of values of $s$? (Here $r(H_1,H_2)$ denotes the usual Ramsey number.) 
In this section, we study the case where only one of $H_1$, $H_2$ is a clique. We obtain lower bounds on $r_{K_s}(H_1,K_t)$ when $t\ge 2$ and $s\le t$, under certain conditions on $H_1$.

In recent work of Conlon, Mattheus, Mubayi, and Verstraëte \cite{conlon2024cycle}, similar techniques as in \cite{mattheus2024} are used to show lower bounds for the cycle-complete Ramsey numbers $r(C_5, K_t)$ and $r(C_7, K_t)$. They derived these bounds through a more general result conditioned on the existence of a graph with certain properties.

Before stating this result, let us recall some definitions from \cite{conlon2024cycle}. Given a graph $F$, let $\mathcal{E}(F)$ denote the set of sequences of edge-disjoint bipartite subgraphs $F_1,\dots, F_k\subseteq F$ such that the edges of the subgraphs $F_i$ form a partition of the edges of $F$, each $F_i$ has at least one edge, and each pair $F_i\neq F_j$ shares at most one vertex. Given $H = (F_1,\dots, F_k)\in \cE(F)$, define $J(H)$ to be the bipartite graph with parts $[k]$ and $V(F)$, where each $i\in [k]$ is adjacent to the vertices $V(F_i)\subseteq V(F)$. Then let $\cL(F) = \{ J(H): H\in \cE(F)\}\cup \{C_4\}$. Finally, for $m\le n$ and $a\le b$, define an $(m,n,a,b)$-graph to be a bipartite graph with parts of size $m$ and $n$ such that every vertex in the part with size $n$ has degree $a$, and every vertex in the part with size $m$ has degree $b$. Now we are ready to state the main result in \cite{conlon2024cycle}.

\begin{theorem}[\cite{conlon2024cycle}]\label{thm:FKt}
    Let $F$ be a graph and let $a,b,m,n$ be positive integers with $a\ge 2^{12}(\log n)^3$ such that there exists an $\mathcal{L}(F)$-free $(m,n,a,b)$-graph. If $t_0=2^8 n(\log n)^2/ab$, then
    \[
        r(F,K_{t_0})=\Omega\left(\frac{bt_0}{\log n}\right).
    \]
\end{theorem}

Here, we obtain an analogous bound on $r_{K_s}(F,K_t)$ in terms of the lower bound in \cref{thm:FKt}.

\begin{theorem} \label{thm:sFKt}
    Let $F$ be a graph and let $a,b,m,n$ be positive integers with $a\ge 2^{12}(\log n)^3$ such that there exists an $\mathcal{L}(F)$-free $(m,n,a,b)$-graph. If $t=2^{11} n(\log n)^2/ab$, then for all $s\le t$,
    \[
        r_{K_s}(F,K_t)=\Omega\left(\binom{g(F,t)}{s}\right),
    \]
    where $g(F,t) = \Omega(\frac{bt_0}{\log n})$ with $t_0=t/8$ is the lower bound from \cref{thm:FKt}.
\end{theorem}

Similarly to the proof of \cref{thm:rs4t}, we use the following technical result proven in \cite{conlon2024cycle}.

\begin{lemma}[\protect{\cite[Lemma~2]{conlon2024cycle}}] \label{lem:Rft}
    Let $F$ be a graph and let $a,b,m,n$ be positive integers with $a\ge 2^{12}(\log n)^3$ such that there exists an $\mathcal{L}(F)$-free $(m,n,a,b)$-graph. Then there exists an $F$-free graph $H$ on $n$ vertices such that each $X\subseteq V(H)$ with $|X|\ge 2^{10}(m\log n)/a$ has 
    \[e(H[X])\ge \frac{a^2}{2^8m}|X|^2.\]
\end{lemma}

\begin{proof}[Proof Sketch of \cref{thm:sFKt}]
    As before, we first prove the case where $s=t=2^{11}n(\log n)^2/ab$. Applying \cref{prop:countindeps} to the graph $H$ given by \cref{lem:Rft} with $R = 2^{10}m(\log n)/a$, $r = t/(2^3\log n)$, and $\alpha = \frac{a^2}{2^8m}$, we conclude that for any $t'\in [t/8, t]$, the number of independent sets of size $t'$ in $H$ is at most
    \[\binom{n}{r}\binom{R}{t'-r}\le n^r\lpr{\frac{eR}{t'}}^{t'}\le \lpr{\frac{2^{10}e^2m\log n}{at'}}^{t'},\]
    where the second inequality uses the fact that $n^r\le n^{t'/\log n}=e^{t'}$.

    Let $G$ be an $(F, K_t)$-Ramsey graph and suppose that $G$ contains $N$ copies of $K_t$ as subgraphs. Taking the same uniformly random map $\pi:V(G)\to V(H)$ as before, we color an edge of $G$ red if it gets mapped to an edge in $H$, and blue otherwise. By the same argument and calculations as before, we can show that 
    \[N\ge \binom{\Omega\lpr{g(F,K_t)}}{t}.\]
    Applying the Kruskal-Katona theorem then yields the desired bound on $r_{K_s}(F,K_t)$ for all $s\le t$.
\end{proof}

Noting that the conditions of \cref{thm:FKt} are satisfied for $F=C_5$ and $F=C_7$ for appropriate values of $(m,n,a,b)$, the authors of \cite{conlon2024cycle} derive the following bounds on the cycle-complete Ramsey numbers in those two cases.

\begin{theorem}[\cite{conlon2024cycle}] \label{thm:cycle-complete}
    As $t\to \infty$,
    \[r(C_5,K_t) = \Omega\lpr{\frac{t^{10/7}}{(\log t)^{13/7}}}, \quad \text{and} \quad r(C_7,K_t) = \Omega\lpr{\frac{t^{5/4}}{(\log t)^{3/2}}}.\]
\end{theorem}

In the same manner, \cref{thm:sFKt} yields the following bounds.

\begin{theorem}
    As $t\to \infty$, for every $s\le t$ we have
    \[r_s(C_5,K_t) = \binom{\Omega\lpr{\frac{t^{10/7}}{(\log t)^{13/7}}}}{s}, \quad \text{and} \quad r_s(C_7,K_t) = \binom{\Omega\lpr{\frac{t^{5/4}}{(\log t)^{3/2}}}}{s}.\]
\end{theorem}

\noindent {\bf Acknowledgments.} The authors would like to thank Lisa Sauermann for helpful discussions about these problems and comments on the paper.

\printbibliography

@article{fox2023triangle,
      title={Triangle {R}amsey numbers of complete graphs}, 
      author={Jacob Fox and Jonathan Tidor and Shengtong Zhang},
      year={2023},
      eprint={2312.06895},
      archivePrefix={arXiv}
}

@article {mattheus2024,
    AUTHOR = {Mattheus, Sam and Verstraete, Jacques},
     TITLE = {The asymptotics of {$r(4,t)$}},
   JOURNAL = {Ann. of Math. (2)},
  FJOURNAL = {Annals of Mathematics. Second Series},
    VOLUME = {199},
      YEAR = {2024},
    NUMBER = {2},
     PAGES = {919--941},
      ISSN = {0003-486X,1939-8980},
   MRCLASS = {05D10 (05B25 05D40)},
  MRNUMBER = {4713025},
       DOI = {10.4007/annals.2024.199.2.8},
}

@book {Lov93,
    AUTHOR = {Lov\'{a}sz, L\'{a}szl\'{o}},
     TITLE = {Combinatorial problems and exercises},
   EDITION = {Second},
 PUBLISHER = {North-Holland Publishing Co., Amsterdam},
      YEAR = {1993},
     PAGES = {635},
      ISBN = {0-444-81504-X},
   MRCLASS = {05-01},
  MRNUMBER = {1265492},
}

@article{conlon2024cycle,
      title={{R}amsey numbers and the {Z}arankiewicz problem}, 
      author={David Conlon and Sam Mattheus and Dhruv Mubayi and Jacques Verstraëte},
      year={2024},
      eprint={2307.08694},
      archivePrefix={arXiv}
}

@article {KimR3t,
    AUTHOR = {Kim, Jeong Han},
     TITLE = {The {R}amsey number {$R(3,t)$} has order of magnitude
              {$t^2/\log t$}},
   JOURNAL = {Random Structures Algorithms},
  FJOURNAL = {Random Structures \& Algorithms},
    VOLUME = {7},
      YEAR = {1995},
    NUMBER = {3},
     PAGES = {173--207},
      ISSN = {1042-9832,1098-2418},
   MRCLASS = {05C55},
  MRNUMBER = {1369063},
MRREVIEWER = {Pavel\ Valtr},
       DOI = {10.1002/rsa.3240070302},
}

@article{AKS80,
title = {A note on Ramsey numbers},
journal = {Journal of Combinatorial Theory, Series A},
volume = {29},
number = {3},
pages = {354-360},
year = {1980},
issn = {0097-3165},
doi = {10.1016/0097-3165(80)90030-8},
author = {Miklós Ajtai and János Komlós and Endre Szemerédi}
}

@article{reiher2023girth,
  title={The girth {R}amsey theorem},
  author={Reiher, Christian and R{\"o}dl, Vojt{\v{e}}ch},
  year={2023},
  eprint={2308.15589},
  archivePrefix={arXiv}
}

@article {RR1995,
    AUTHOR = {R\"odl, Vojt\v ech and Ruci\'nski, Andrzej},
     TITLE = {Threshold functions for {R}amsey properties},
   JOURNAL = {J. Amer. Math. Soc.},
  FJOURNAL = {Journal of the American Mathematical Society},
    VOLUME = {8},
      YEAR = {1995},
    NUMBER = {4},
     PAGES = {917--942},
      ISSN = {0894-0347,1088-6834},
   MRCLASS = {05C55 (05C80 05D10)},
  MRNUMBER = {1276825},
       DOI = {10.2307/2152833}
}
\end{document}